\documentclass[11pt]{article}

\usepackage{tikz}
\usepackage{subfigure}
\usepackage{scalefnt}
\usepackage{bm}

\usepackage{amsmath,amsfonts,amssymb,amsthm,enumerate,graphicx}

\usepackage{hyperref}

\usepackage{tikz}
\usetikzlibrary{arrows,backgrounds}
\usetikzlibrary{snakes}
\usetikzlibrary{automata}
\usetikzlibrary{er}
\usetikzlibrary{folding}
\usetikzlibrary{matrix}
\usetikzlibrary{mindmap}
\usetikzlibrary{petri}
\usetikzlibrary{plothandlers}
\usetikzlibrary{plotmarks}
\usetikzlibrary{shapes}
\usetikzlibrary{shapes.symbols}
\usetikzlibrary{shapes.multipart}
\usetikzlibrary{shapes.misc}
\usetikzlibrary{topaths}
\usetikzlibrary{trees}

\numberwithin{equation}{section}

\newtheorem{theo}{{\bf Theorem}}

\newtheorem{prop}{{\bf Proposition}}[section]
\newtheorem{lemma}[prop]{{\bf Lemma}}
\newtheorem{claim}{{\bf Claim}}

\usepackage{mathrsfs}
\usepackage{color}
\usepackage{graphics}
\usepackage{epsfig}
\usepackage{tikz,authblk} 

\definecolor{DRed}{rgb}{0.5,0.05,0.05}
\definecolor{DGreen}{rgb}{0.05,0.5,0.05}
\definecolor{Melon}{rgb}{0.7,0.1,0}

\definecolor{DGrey}{rgb}{0.3,0.3,0.3}
\definecolor{PBlue}{rgb}{0.45,0.65,1.0}
\definecolor{SBlue}{rgb}{0.33,0.49,0.75}
\definecolor{DBlue}{rgb}{0.05,0.05,0.50}
\definecolor{PrBlue}{rgb}{0,0.23,0.41}

\definecolor{Wine}{rgb}{0.5,0,0.05}
\definecolor{PrBlueTx}{rgb}{0,0.20,0.47}
\definecolor{HeadBlu}{rgb}{0,0.22,0.44}

\begin{document}

\title{Isoperimetric inequality for non-Euclidean polygons}
%\title{An isoperimetric inequality in the hyperbolic plane}

\author[1] {Basudeb Datta}
\author[2] {Subhojoy Gupta}

\affil[1]{Institute for Advancing Intelligence (IAI), TCG CREST,  Bidhannagar, Kolkata 700\,091, India.
 bdatta17@gmail.com.}
\affil[2]{Department of Mathematics, Indian Institute of Science, Bangalore 560\,012, India. 
subhojoy@iisc.ac.in.}
%\vspace{-15mm}
%\date{January 06, 2023}
%\date{April 08, 2023}
\date{August 27, 2024}

\maketitle

\vspace{-7.5mm}

\begin{abstract}
%In 1964, T\'{o}th showed that among the spherical polygons of equal area having at most $n$ sides the regular $n$-gon has the least possible perimeter and stated that this result on spherical polygons can be extended without difficulty to hyperbolic polygons. 
%In 1984, Bezdek stated, without proof, that among all $n$-gons of fixed perimeter in the hyperbolic plane, the regular $n$-gon has the largest area. 
It is a classical fact in Euclidean geometry that the regular polygon maximizes area amongst polygons of the same perimeter and number of sides, and the analogue of this in non-Euclidean geometries has long been a folklore result. 
In this note, we present a complete proof of this  polygonal isoperimetric inequality in hyperbolic and spherical geometries.
\end{abstract}

\noindent {\small {\em MSC 2020\,:}  52A10; 52A38; 52A55. 

\noindent {\em Keywords:} Convex polygons, hyperbolic geometry, spherical geometry, isoperimetric inequalities.}

\section{Introduction}

In 1964, T\'{o}th provided a proof of the classical fact that  {\em among all $n$-gons of fixed perimeter in the Euclidean plane, the regular $n$-gon has the largest area} (cf. \cite[Page 160]{T}). In the same book (see Pages 213, 214) T\'{o}th showed (sketching the main steps of a proof) that among the spherical polygons of equal area having at most $n$ sides the
regular $n$-gon has the least possible perimeter. Therefore, {\em among all $n$-gons of fixed perimeter in an open hemisphere, the regular $n$-gon has the largest area}. 
He also stated (cf. \cite[Page 256]{T}) that this result on spherical polygons can be extended without difficulty to hyperbolic polygons. 
In \cite{B}, Bezdek gave a proof of this isoperimetric inequality in the Euclidean plane; he also stated (without proof) that the similar result is true for the hyperbolic plane. However, the proof for hyperbolic case seems does not follow by exactly the same arguments as Bezdek's. Our search for a complete proof of this non-Euclidean isoperimetric inequality in the literature yielded the following two more recent articles:
\begin{itemize}
\item  In \cite[Proposition 3.7]{HLPX} the proof assumes the fact that the circle is the unique curve realizing the equality in the isoperimetric inequality, and 
\item In  the preprint  \cite{AC}  a proof  is provided for all constant curvature space forms,  based on the idea in \cite[section 3]{HHM}.
\end{itemize} 
Our goal in this note is to remedy this gap in the literature, and provide a new and  detailed  proof of this isoperimetric inequality for hyperbolic and spherical polygons. One feature of our proof that we would like to highlight is that the steps are elementary and should be easy to translate into a computer proof system.

Throughout this article, a hyperbolic polygon will have vertices in the hyperbolic plane, and a spherical polygon will have vertices in an open hemisphere. In both cases, the sides of a polygon are geodesic segments in the respective geometries, and a polygon with $n$ sides is called an $n$-gon. Also, a polygon is called \textit{regular} if the lengths of the sides are all equal (also called \textit{equilateral}), and the interior angles at each vertex are also all equal (also called \textit{equiangular}). In this note, we  present a proof of the following.

\begin{theo} \label{peri&area}
For $n\geq 3$, among all hyperbolic (respectively, spherical) $n$-gons  of fixed perimeter, the regular hyperbolic  (respectively, spherical) $n$-gon has the largest area. 
\end{theo}

Our proof follows a strategy similar to the  proof of the isoperimetric inequality for Euclidean polygons provided by  Mossinghoff  in \cite{M}. In particular, we shall rely on some results in a relatively recent article by Wimmer in \cite{W} that generalize a classical result concerning maximizing the area of quadrilaterals. We shall present these results, together with a brief exposition of the geometry of the hyperbolic plane and the sphere, in the next section. We remark that there is another classical isoperimetric problem that seeks the convex Euclidean $n$-gon of given perimeter that \textit{minimizes the diameter} (see \cite{D}); this has not been investigated in non-Euclidean geometries.

\bigskip

\noindent \textbf{Acknowledgements.}  This work was supported by the Department of Science and Technology, Govt.of India grant no. CRG/2022/001822, and by the DST FIST program - 2021 [TPN - 700661].

%\newpage

\section{A brief introduction to non-Euclidean geometries}

\subsection{Hyperbolic geometry}

A well-known model of the hyperbolic plane is the \textit{Poincar\'{e} disk model}, which is the open unit disk in $\mathbb{R}^2$ equipped with the Riemannian metric $4(dr^2 + r^2 d\theta^2)/(1-r^2)^2$ (in polar coordinates). The \textit{boundary at infinity} is the unit circle in this model. 

The Poincar\'{e} disk model is \textit{conformal} since the metric written above is conformal to the Euclidean metric; in particular, angles in the hyperbolic metric are the same as the usual angles in the Euclidean plane.

%The two most well-known models of the hyperbolic plane is the \textit{Poincar\'{e} model}, which is the open unit disk in $\mathbb{R}^2$ equipped with the Riemannian metric $4(dr^2 + r^2 d\theta^2)/(1-r^2)^2$ (in polar coordinates), and the \textit{upper half-plane model}, which is the upper half-plane $\{(x,y) \in \mathbb{R}^2 \ \vert\ y>0\}$ equipped with the Riemannian metric $(dx^2 +dy^2)/y^2$ (in Cartesian coordinates). The \textit{boundary at infinity} is the unit circle in the Poincar\'{e} model, and the extended real line $\{(x,0)\ \vert \ x \in \mathbb{R}\} \cup \{\infty\}$ in the upper half-plane model. 

% These two models are related by a conformal isometry, which is in fact a linear fractional transformation when the domain $\mathbb{R}^2$ is identified with the complex plane $\mathbb{C}$.  In fact, these are both \textit{conformal} models since the metrics written above is conformal to the Euclidean metric; in particular, angles in the hyperbolic metric are the same as the usual angles in the Euclidean plane. 
 
Geodesics between points are distance-minimizing paths; in the Poincar\'{e} disk these are segments of semi-circular arcs that are perpendicular to the boundary circle, including diameters.
%, and in the upper half-plane model they are segments of semi-circles that are perpendicular to the boundary line, together with vertical lines. 
Such geodesic segments will be the sides of the polygons that we consider in the hyperbolic plane.

A few more notions that will be relevant to this article are the following: 
\begin{itemize}

\item a \textit{circle} in the hyperbolic plane is the set of points at equal hyperbolic distance from a point (the \textit{hyperbolic center}); these are also Euclidean circles in the Poincar\'{e} disk, although the Euclidean centers and hyperbolic centers do not coincide except for circles centered at the origin. 
\item a \textit{horocycle} in the hyperbolic plane in these models is a circle that is tangent to the boundary at infinity. (See Fig. \ref{fig:hypocycle}.) 
\item a \textit{hypercycle} in the hyperbolic plane is a circular arc that intersects the boundary at infinity at both endpoints at an angle $\alpha$. When $\alpha = \pi/2$, such an arc is a geodesic line, and in general it is equidistant from the geodesic line with the same endpoints on the boundary at infinity. (See Fig. \ref{fig:hypercycle}.) 
\end{itemize}

The group of isometries of the hyperbolic plane is generated by reflections along geodesic lines, and any orientation-preserving isometry that is not the identity map falls in the following three types: an \textit{elliptic} isometry fixes a point in the hyperbolic plane  and preserves any circle centered at that point (rotating the circle), a \textit{parabolic} isometry fixes a point on the boundary at infinity, and preserves any horocycle centered at that point (translating along it), and a \textit{hyperbolic} isometry fixes a pair of distinct points on the boundary and preserves any hypercycle with those endpoints (translating along it).

 Note that in contrast with Euclidean geometry, the sum of interior angles of a hyperbolic triangle are strictly less than $\pi$. For more comparisons, and historical context, see for example \cite{A}.

\subsection{Spherical geometry}

The usual model for spherical geometry is the unit sphere in $\mathbb{R}^3$ equipped with the metric induced from the Euclidean metric in $\mathbb{R}^3$. Geodesics in this model are segments of great circles, which are circles that are intersections of the sphere with planes in $\mathbb{R}^3$ that pass through the origin. As mentioned in the introduction, the spherical polygons we shall consider, shall have vertices that lie in an open hemisphere;  the segment of the great circle in that hemisphere  between any pair of points is the unique distance-minimizing path between them.  Once again, reflections along great circles generate the group of isometries. As for the hyperbolic plane, a  \textit{circle} in spherical geometry is the set of points at equal distance from a point in the open hemisphere. 

As is well known, the sum of interior angles of a spherical triangle are strictly greater than $\pi$. However, in both the spherical and hyperbolic geometries, the following fact from Euclidean geometry still holds. We provide the (elementary) proof as not only shall we use the fact, we shall modify the proof to extend to the case of \textit{ideal} hyperbolic triangles in \S3. 

\begin{lemma}\label{lem1} The base angles of a (spherical or hyperbolic) isosceles triangle are equal.
\end{lemma}

\begin{proof}
Let $uvw$ be the triangle, such that the lengths of sides $uv$ and $vw$ are equal. Then, in both geometries, there is an isometry (a reflection) that fixes $v$ and interchanges $u$ and $w$. Note that such an isometry will interchange the equal sides, and preserve the geodesic side $uw$, flipping its orientation. In particular, the isometry takes the angle $\angle vwu$ to $\angle vuw$, and since an isometry preserves angles, these two angles are equal. 
\end{proof}

\subsection{Areas of non-Euclidean triangles}

%By a spherical  polygon we mean a polygon in an open hemisphere whose edges are great circle arcs. By a hyperbolic polygon, we mean a polygon in the (open) Poincar\'{e} disk whose edges are geodesics. 

The following Lhuilier's formula for the spherical triangle is known in the 19th century itself (cf. \cite[Page 36, Part II]{MP}). 

\begin{prop} \label{Lhuilier}
Let $T$ be a spherical triangle in an open hemisphere. If the  side lengths of $T$ are $a, b, c$, then its area $S= \mbox{area}(T)$ is given by:
\begin{align*}
\tan^2\frac{S}{4} &= \tan\frac{s}{2}\tan\frac{s-a}{2}\tan\frac{s-b}{2}\tan\frac{s-c}{2}, 
\end{align*}
where $s = (a+b+c)/2$ is the semi-perimeter of $T$.
\end{prop}

Similar to Lhuilier's formula, the following formula (cf. \cite[Page 152]{AM}) gives the area of a hyperbolic triangle in terms of the lengths of the sides. 

\begin{prop} \label{Heron}
Let $T$ be a hyperbolic triangle with side lengths $a, b, c$.  Then its area $S= \mbox{area}(T)$ is given by:
\begin{align*}
\tan^2\frac{S}{2} &= \tanh\frac{s}{2}\tanh\frac{s-a}{2}\tanh\frac{s-b}{2}\tanh\frac{s-c}{2}, 
\end{align*}
where $s = (a+b+c)/2$ is the semi-perimeter of $T$.
\end{prop}

These can be thought of as the non-Euclidean analogues of the familiar Heron's formula for the area of a Euclidean triangle:  $S^2 = s(s-a)(s-b)(s-c)$.

\subsection{Areas of non-Euclidean quadrilaterals} 

A polygon (in any of the geometries) is said to be \textit{cyclic} if its vertices lie on a circle, in this case we  also say the polygon is inscribed in the circle. It is a well-known fact that \textit{amongst Euclidean quadrilaterals with given sides, the quadrilateral of maximum area is cyclic}. One proof of this uses  a formula for the area of  quadrilaterals in terms of the side-lengths and the average of opposite interior angles that generalizes Brahmagupta's formula for the area of cyclic quadrilaterals (see Page 387 of \cite{M}). 

The article \cite{W} of Wimmer was about the non-Euclidean generalizations of this, that will be key in our proof. In particular, he presented a proof of the following: 

\begin{prop}[Theorem 1 in \cite{W}] 
\label{W1}
Among all non-Euclidean  (i.e., spherical and hyperbolic) quadrilaterals of given sides, a quadrilateral $Q$ with internal angles $A, B, C, D$ (in a cyclic order) has largest area if and only if $A+C = B+D$. 
\end{prop}

Moreover, he showed that in the spherical case, this opposite-angle condition is equivalent to the quadrilateral being cyclic:

\begin{prop}[Lemmas 2 \& 3 in \cite{W}] 
\label{W2}
A spherical convex  quadrilateral $Q$ with internal angles $A, B, C, D$ (in a cyclic order) is cyclic if and only if $A+C = B+D$. 
\end{prop}

Although in the hyperbolic case, the equivalence fails, he showed that the following holds: 

\begin{prop}[Lemma 4 in \cite{W}] 
\label{W3}
Let $Q$ be a convex quadrilateral in the hyperbolic disc  with internal angles $A, B, C, D$ (in a cyclic order). If  $A+C = B+D$ then $Q$ is inscribed into a circle, a horocycle, or a hypercycle.
\end{prop}

%Note that it is $\tan^2\frac{S}{4}$ and not  $\tan^2\frac{S}{2}$. 

\section{Proof of Theorem \ref{peri&area}}

\begin{proof}[Proof of Theorem \ref{peri&area}] 
 Fix $n\geq 3$. 
 
 \medskip
 
 \noindent {\bf Hyperbolic case}: 
Let $X$ be a hyperbolic $n$-gon  of perimeter $L$ in the open Poincar\'{e} disk $D$ and of largest area among all hyperbolic $n$-gons  of perimeters $L$ in $D$. 

\begin{claim} \label{cl:h1}
{\rm $X$ is equilateral. }
\end{claim}

If possible, suppose $X$ is not equilateral. Then there exist two adjacent sides $uv$ and $vw$ of different length. Say, $\mbox{length}(uv) = a < b = \mbox{length}(vw)$. Let $\mbox{length}(uw) = c$. 

%%%%%%%%

\begin{figure}[ht!]

\tikzstyle{ver}=[]
\tikzstyle{vert}=[circle, draw, fill=black!100, inner sep=0pt, minimum width=4pt]
\tikzstyle{vertex}=[circle, draw, fill=black!00, inner sep=0pt, minimum width=4pt]
\tikzstyle{edge} = [draw,thick,-]
\centering

\begin{tikzpicture}[scale=0.25]

\draw (8, 12) .. controls (9, 12.4) and (11.5, 13) .. (13.5, 15) .. controls (13.5, 15) and (16, 17) .. (17,20); 

\draw (32, 12) .. controls (29, 12.5) and (26, 13) .. (22.5, 15) .. controls (22.5, 15) and (19, 17) .. (17,20); 

\draw (8, 12) .. controls (12.5, 13) and (17, 13.5) .. (21.5, 13) .. controls (23, 13) and (27.5, 12.4) .. (32,12); 

\node[ver] () at (8.2, 12.1) {\small $\bullet$}; 
\node[ver] () at (17.1, 20) {\small $\bullet$};  
\node[ver] () at (32.1, 12) {\small $\bullet$}; 

 \node[ver] () at (8, 13.5) {$u$};  \node[ver] () at (18.5, 20) { $v$};  \node[ver] () at (32, 13.2) { $w$}; 
 
 \node[ver] () at (13.5, 16.5) {$a$};  \node[ver] () at (23, 16.5) {$b$};  
 \node[ver] () at (20, 12) {$c$}; 

%%%%------------------------

\draw (36, 12) .. controls (38, 12.7) and (40, 13.5) .. (42.5, 15) .. controls (42.5, 15) and (46, 17) .. (48,20); 

\draw (60, 12) .. controls (58, 12.4) and (56, 13.5) .. (53.5, 15) .. controls (53.5, 15) and (50, 17) .. (48,20); 

\draw (36, 12) .. controls (40.5, 13) and (45, 13.5) .. (49.5, 13) .. controls (51, 13) and (55.5, 12.4) .. (60,12); 

\node[ver] () at (36.4, 12.1) {\small $\bullet$}; 
\node[ver] () at (48.1, 20) {\small $\bullet$};  
\node[ver] () at (60, 12) {\small $\bullet$}; 

 \node[ver] () at (36, 13.5) {$u$};  \node[ver] () at (50, 20) { $v^{\,\prime}$};  \node[ver] () at (60.5, 13.2) { $w$}; 
 
\node[ver] () at (42, 16.5) {$a^{\,\prime}$};  \node[ver] () at (54.5, 16.5) {$b^{\,\prime}$};  \node[ver] () at (48, 12) {$c$};

\end{tikzpicture} 

\caption{Hyperbolic triangles} \label{fig:triangle_h}
\end{figure}
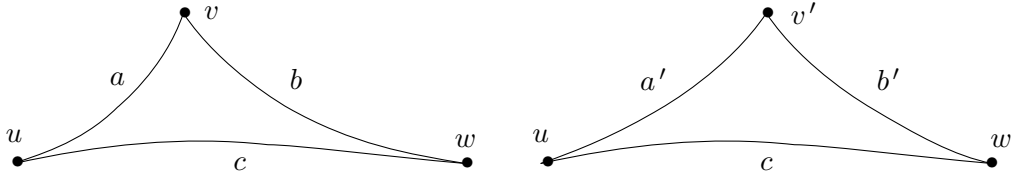 

Choose $v^{\,\prime}$ such that $a^{\prime} := \mbox{length}(uv^{\,\prime}) = (a + b)/2 = \mbox{length}(v^{\,\prime}w) =: b^{\prime}$. (The line $uw$ divides $D$ and we choose  $v^{\,\prime}$ so that $v$ and $v^{\,\prime}$ are on the same  half. See Fig. \ref{fig:triangle_h}.)  
Let $X^{\prime}$ be the $n$-gon whose vertices are those of $X$ other than $v$ together with $v^{\,\prime}$. 
Let $S$ be the area of the triangle $uvw$ and $S^{\,\prime}$ be the area of the triangle $uv^{\,\prime}w$. Let $s = (a + b + c)/2 = (a^{\prime} + b^{\,\prime}  + c)/2$. Since $(s-b)/2 < (s-a)/2$, we have 
\begin{align*} 
\tanh\frac{s-b}{2}\tanh\frac{s-a}{2}  < &\tanh^2 \frac{(s-a)/2+(s-b)/2}{2}  \\
 &= \tanh^2\frac{s-a^{\prime}}{2} 
 = \tanh\frac{s-a^{\prime}}{2}\tanh\frac{s-b^{\prime}}{2}. 
\end{align*} 
Therefore, by the hyperbolic Heron's formula (see Proposition \ref{Heron}),  $\tan^2\frac{S}{2} <  \tan^2\frac{S^{\,\prime}}{2}$.  This implies that $S< S^{\,\prime}$ and hence $\mbox{area}(X) < \mbox{area}(X^{\prime})$. This is not possible since $\mbox{peri}(X^{\prime}) = \mbox{peri}(X)=L$ and $X$ has the largest area with perimeter $L$. This proves Claim \ref{cl:h1}. 

\medskip

Let $X = v_1\cdots v_n$. For $1\leq i\leq n$, consider  the quadrilateral $Q:= v_{i-1}v_iv_{i+1}v_{i+2}$  (addition and subtraction in the subscripts are modulo $n$).  By Claim \ref{cl:h1}, $\mbox{length}(v_{i-1}v_{i}) = \mbox{length}(v_{i}v_{i+1}) = \mbox{length}(v_{i+1}v_{i+2}) = \ell$ (say). Let $\mbox{length}(v_{i-1}v_{i+2}) = k$. Let the internal angles of $Q$ be $A, B, C, D$ as in  Fig. \ref{fig:quardangle} (a).

%\bigskip

\begin{claim} \label{cl:h3} 
$B=C$. 
\end{claim}

%\medskip

Since the convex $n$-gon has more area than non-convex $n$-gon of same sides (and hence having same perimeter), we now assume that $X$ is convex. This implies that $Q$ is convex. Since  $X$ has the largest area with perimeter $L$, it follows that $Q$  has the largest area with given sides. Therefore, by Propositions \ref{W1} \& \ref{W3},  $Q$ is inscribed into a circle, a horocycle, or a hypercycle. 

\smallskip

\noindent {\sf Case 1.} $Q$ is inscribed in a circle. So, the vertices of $Q$ are on a circle. Let $o$ be the (hyperbolic) centre of this circle. Then $\mbox{length}(ov_{i-1}) =  \mbox{length}(ov_{i}) = \mbox{length}(ov_{i+1}) = \mbox{length}(ov_{i+2})$. Thus, the isosceles triangles $ov_{i-1}v_{i}$, $ov_{i}v_{i+1}$ and $ov_{i+1}v_{i+2}$ are pairwise congruent. Moreover, by Lemma \ref{lem1} the base angles of each of these isosceles triangles are equal. Therefore,  all the six angles $\angle ov_{i-1}v_{i}$, $\angle ov_{i}v_{i-1}$,  $\angle ov_{i}v_{i+1}$, $\angle ov_{i+1}v_{i}$, $\angle ov_{i+1}v_{i+2}$,  $\angle ov_{i+2}v_{i+1}$ are same, say equal to $\theta$. This implies that $B= 2\theta = C$ (see Fig. \ref{fig:quardangle} (b)).

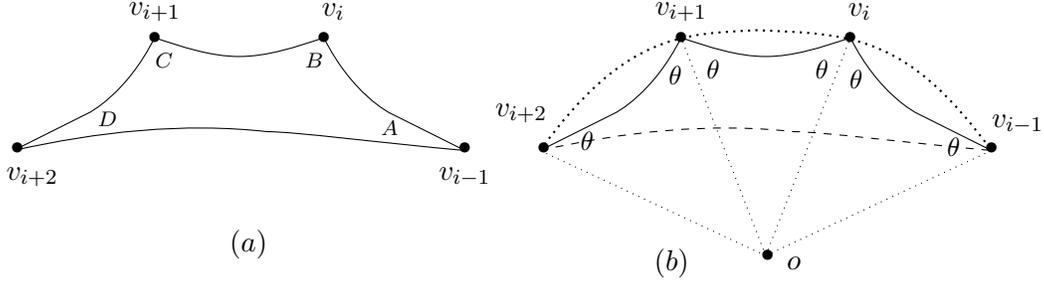
\begin{figure}[ht!]

\tikzstyle{ver}=[]
\tikzstyle{vert}=[circle, draw, fill=black!100, inner sep=0pt, minimum width=4pt]
\tikzstyle{vertex}=[circle, draw, fill=black!00, inner sep=0pt, minimum width=4pt]
\tikzstyle{edge} = [draw,thick,-]
\centering

\begin{tikzpicture}[scale=0.25]

\draw (8, 12) .. controls (10, 13) and (11, 13.5) .. (12, 14) .. controls (12, 14) and (14, 15) .. (15.5,18); 

\draw (32, 12) .. controls (30, 13) and (29, 13.5) .. (28, 14) .. controls (28, 14) and (26, 15) .. (24.5,18); 

\draw (15.5, 18) .. controls (17, 17.5) and (18.5, 17) .. (20, 17) .. controls (21.5, 17) and (23, 17.5) .. (24.5,18);

\draw (8, 12) .. controls (12.5, 13) and (17, 13.5) .. (21.5, 13) .. controls (23, 13) and (27.5, 12.4) .. (32,12); 

\node[ver] () at (8.2, 12.1) {\small $\bullet$}; 
\node[ver] () at (32, 12.1) {\small $\bullet$}; 
\node[ver] () at (15.5, 18) {\small $\bullet$}; 
\node[ver] () at (24.5, 18) {\small $\bullet$}; 

 \node[ver] () at (9, 10.6) {$v_{i+2}$};  \node[ver] () at (15.5, 19.4) { $v_{i+1}$}; 
 \node[ver] () at (25, 19.4) {$v_{i}$};  \node[ver] () at (32, 10.6) { $v_{i-1}$}; 
 
\node[ver] () at (13, 13.7) {\scriptsize $D$};  \node[ver] () at (16, 16.8) {\scriptsize  $C$};  \node[ver] () at (24, 16.8) {\scriptsize $B$};  \node[ver] () at (28, 13.2) {\scriptsize $A$};

\node[ver] () at (20.5, 7) {$(a)$}; 
%%%------------------------

\draw (36, 12) .. controls (38, 13) and (39, 13.5) .. (40, 14) .. controls (40, 14) and (42, 15) .. (43.5,18); 

\draw (60, 12) .. controls (58, 13) and (57, 13.5) .. (56, 14) .. controls (56, 14) and (54, 15) .. (52.5,18); 

\draw (43.5, 18) .. controls (45, 17.5) and (46.5, 17) .. (48, 17) .. controls (49.5, 17) and (51, 17.5) .. (52.5,18);

\draw[dashed] (36, 12) .. controls (40.5, 13) and (45, 13.5) .. (49.5, 13) .. controls (51, 13) and (55.5, 12.4) .. (60,12); 

\draw[dotted, thick] (36, 12) .. controls (40.5, 18.2) and (43.55, 18.2) .. (47.5, 18.4) .. controls (52.5, 18.2) and (55.5, 18.2) .. (60,12); 

\draw[dotted]  (48, 6.4) -- (36, 12); \draw[dotted]  (48, 6.4) -- (43.5, 18); 
 \draw[dotted]  (48, 6.4) -- (52.5, 18); \draw[dotted]  (48, 6.4) -- (60, 12);

\node[ver] () at (36.2, 12.1) {\small $\bullet$}; 
\node[ver] () at (60, 12.1) {\small $\bullet$}; 
\node[ver] () at (43.5, 18) {\small $\bullet$}; 
\node[ver] () at (52.5, 18) {\small $\bullet$}; 
\node[ver] () at (48.1, 6.4) {\small $\bullet$}; 

 \node[ver] () at (35, 14) {$v_{i+2}$};  \node[ver] () at (43.5, 19.4) { $v_{i+1}$}; 
 \node[ver] () at (53, 19.4) {$v_{i}$};  \node[ver] () at (61.4, 13.4) { $v_{i-1}$}; 
 \node[ver] () at (49.5, 6) {$o$}; 
 
 \node[ver] () at (38.5, 12.5) {\small $\theta$};  \node[ver] () at (43.2, 16) {\small  $\theta$};  \node[ver] () at (45.2, 16.3) {\small  $\theta$};  
 
 \node[ver] () at (51, 16.3) {\small $\theta$};  \node[ver] () at (52.8, 15.8) {\small $\theta$};  \node[ver] () at (58, 12.2) {\small$\theta$}; 
 
 \node[ver] () at (43, 6) {$(b)$}; 
 
\end{tikzpicture} 

\caption{$(a)$ Hyperbolic quadrilateral $Q$, $(b)$ Inscribed in a circle as in Case 1} \label{fig:quardangle}
\end{figure}

\smallskip

\noindent {\sf Case 2.} $Q$ is inscribed in a horocycle. Let $o$ be the point where the horocycle is tangent to the boundary at infinity.  The triangles $ov_{i-1}v_i, ov_iv_{i+1}$ and $ov_{i+1}v_{i+2}$ are now \textit{ideal} triangles since one vertex is at infinity, and the sides incident at $o$ are of infinite length. Since the lengths of the sides opposite to $o$ are all equal, these triangles are congruent as for any pair of such sides,  there is a parabolic isometry that fixes $o$, preserves the horocycle, and takes one side to another. Moreover, for each triangle, say $ov_{i-1}v_i$, the two base angles are equal, namely $\angle o v_{i-1} v_i = \angle o v_iv_{i-1}$, since as in the proof of Lemma \ref{lem1}, there is a reflection isometry that fixes $o$ and interchanges the vertices $v_{i-1}$ and $v_i$ lying on the horocycle (see Fig. \ref{fig:hypocycle} ).  As in Case 1, this implies that $\angle ov_{i-1}v_{i} = \angle ov_{i}v_{i-1}= \angle ov_{i}v_{i+1} = \angle ov_{i+1}v_{i}=\angle ov_{i+1}v_{i+2} = \angle ov_{i+2}v_{i+1} = \theta$ and consequently $B= 2\theta = C$.

\begin{figure}[ht!]

\tikzstyle{ver}=[]
\tikzstyle{vert}=[circle, draw, fill=black!100, inner sep=0pt, minimum width=4pt]
\tikzstyle{vertex}=[circle, draw, fill=black!00, inner sep=0pt, minimum width=4pt]
\tikzstyle{edge} = [draw,thick,-]
\centering

\begin{tikzpicture}[scale=0.17]

\draw (20,17.5) circle (220mm); 
\draw[thick] (20,8) circle (125mm); 

%\draw (20,12) ellipse (120mm and 50mm); 

\draw (8.8, 13.5) .. controls (10.6, 14.5) and (11.5, 15) .. (12.6, 15.5) .. controls (12.6, 15.5) and (14.5, 16.5) .. (15.5,19.5);

\draw (31.3, 13.5) .. controls (29.5, 14.5) and (28.6, 15) .. (27.7, 15.5) .. controls (27.7, 15.5) and (26, 16.5) .. (24.5,19.5); 

\draw (15.5, 19.8) .. controls (17, 19.3) and (18.5, 18.8) .. (20, 18.8) .. controls (21.5, 18.8) and (23, 19.3) .. (24.5,19.8); 

\draw[dotted] (20, -4.5) .. controls (19.8, -0.5) and (19.5, 2.5) .. (19, 6.5) .. controls (18.4, 10.5) and (17.6, 14.5) .. (15.6,19.75);

\draw[dotted] (20, -4.5) .. controls (20.2, -0.5) and (20.5, 2.5) .. (21, 6.5) .. controls (21.6, 10.5) and (22.4, 14.5) .. (24.4,19.75); 

\draw[dotted] (20, -4.5) .. controls (21, -1.5) and (22, 1.5) .. (24, 4.5) .. controls (26, 7.5) and (28, 10.5) .. (31.3,13.5); 

\draw[dotted] (20, -4.5) .. controls (19, -1.5) and (18, 1.5) .. (16, 4.5) .. controls (14, 7.5) and (12, 10.5) .. (8.7,13.5); 

\node[ver] () at (8.8, 13.6) {\small $\bullet$}; \node[ver] () at (31.3, 13.5) {\small $\bullet$}; 
\node[ver] () at (15.65, 19.75) {\small $\bullet$}; \node[ver] () at (24.45, 19.75) {\small $\bullet$}; \node[ver] () at (20.05, - 4.5) {\small $\bullet$}; 

 \node[ver] () at (6.5, 14.6) {$v_{i+2}$};  \node[ver] () at (15.5, 21.5) { $v_{i+1}$}; 
 \node[ver] () at (25, 21.4) {$v_{i}$};  \node[ver] () at (33.8, 14.2) { $v_{i-1}$}; 
 
 \node[ver] () at (21.5, - 3.5) {$o$}; 

 \node[ver] () at (17, 16.5) {\scriptsize  $C$};  \node[ver] () at (23.1, 16.3) {\scriptsize $B$};  

 \node[ver] () at (15.6, 17.6) {\small  $\theta$};  \node[ver] () at (17.5, 18) {\small  $\theta$};  \node[ver] () at (22.8, 18) {\small $\theta$};  \node[ver] () at (24.6, 17.6) {\small $\theta$};  

\node[ver] () at (10.7, 13.2) {\small $\theta$};  \node[ver] () at (29.2, 13.2) {\small $\theta$};  

\end{tikzpicture} 

\caption{In Case 2, when the vertices on a horocycle: each edge determines an ideal triangle.} \label{fig:hypocycle}
\end{figure}
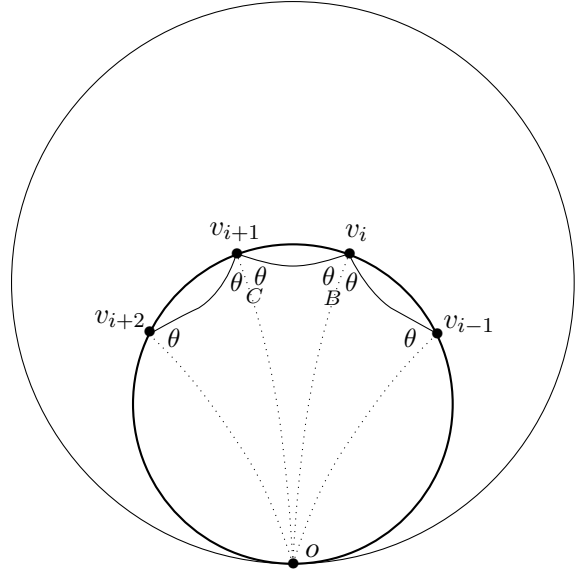 

%%%%%%%%%%%%%%%%%%%%%%%%%%%%

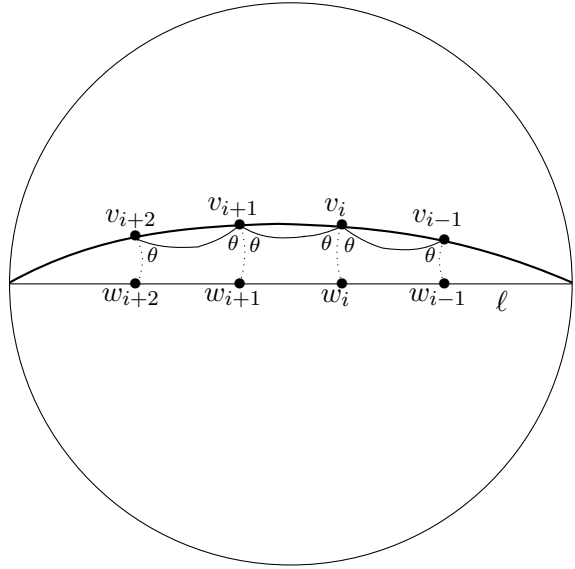
\begin{figure}[ht!]

\tikzstyle{ver}=[]
\tikzstyle{vert}=[circle, draw, fill=black!100, inner sep=0pt, minimum width=4pt]
\tikzstyle{vertex}=[circle, draw, fill=black!00, inner sep=0pt, minimum width=4pt]
\tikzstyle{edge} = [draw,thick,-]
\centering

\begin{tikzpicture}[scale=0.17]

\draw (20,17.5) circle (220mm); \draw  (-2, 17.5) -- (42, 17.5); 

\draw[thick] (-2, 17.6) .. controls (5, 21.5) and (12, 22) .. (19, 22.2) .. controls (26, 22) and (33, 21.5) .. (42,17.6);

%\draw[thick] (20,8) circle (125mm); 

%\draw (20,12) ellipse (120mm and 50mm); 

\draw (8, 21) .. controls (9.6, 20.3) and (11.2, 20.3) .. (12.8, 20.4) .. controls (12.8, 20.5) and (14.4, 20.7) .. (16,22.1); 

\draw (32, 21) .. controls (30.4, 20) and (28.8, 20.1) .. (27.2, 20.3) .. controls (27.2, 20.3) and (25.6, 20.7) .. (24,22);

\draw (16, 22.1) .. controls (17.6, 21) and (19.2, 21) .. (21, 21.2) .. controls (20.8, 21.3) and (21.7, 21) .. (24,22);

\draw[dotted] (8, 17.5) .. controls (8.3, 18.2) and (8.4, 18.9) .. (8.4, 19.6) .. controls (8.4, 19.6) and (8.3, 20.3) .. (8,21); 

\draw[dotted] (16, 17.5) .. controls (16.3, 18.4) and (16.4, 19.3) .. (16.4, 20.2) .. controls (16.4, 20.2) and (16.3, 21.9) .. (16,22); 

\draw[dotted] (24, 17.5) .. controls (23.7, 18.4) and (23.6, 19.3) .. (23.6, 20.2) .. controls (23.6, 20.2) and (23.7, 21.9) .. (24,22); 

\draw[dotted] (32, 17.5) .. controls (31.7, 18.2) and (31.6, 18.9) .. (31.6, 19.6) .. controls (31.6, 19.6) and (31.7, 20.3) .. (32,21);

\node[ver] () at (7.85, 21.2) {\small $\bullet$}; \node[ver] () at (32, 20.9) {\small $\bullet$}; 
\node[ver] () at (16, 22.1) {\small $\bullet$}; \node[ver] () at (24, 22.1) {\small $\bullet$}; 

\node[ver] () at (7.85, 17.5) {\small $\bullet$}; \node[ver] () at (32, 17.5) {\small $\bullet$}; 
\node[ver] () at (16, 17.5) {\small $\bullet$}; \node[ver] () at (24, 17.5) {\small $\bullet$};

 \node[ver] () at (7.5, 22.6) {$v_{i+2}$};  \node[ver] () at (15.5, 23.5) { $v_{i+1}$}; 
 \node[ver] () at (23.5, 23.5) {$v_{i}$};  \node[ver] () at (31.5, 22.6) { $v_{i-1}$}; 
 
  \node[ver] () at (7.5, 16.4) {$w_{i+2}$};  \node[ver] () at (15.5, 16.4) { $w_{i+1}$}; 
 \node[ver] () at (23.5, 16.4) {$w_{i}$};  \node[ver] () at (31.5, 16.4) { $w_{i-1}$}; 
 \node[ver] () at (36.5, 16.2) { $\ell$};

% \node[ver] () at (17, 16.5) {\scriptsize  $C$};  \node[ver] () at (23.1, 16.3) {\scriptsize $B$};  

 \node[ver] () at (9.2, 19.8) {\scriptsize $\theta$};  \node[ver] () at (15.5, 20.7) {\scriptsize  $\theta$};  \node[ver] () at (17.2, 20.6) {\scriptsize  $\theta$};  \node[ver] () at (22.8, 20.7) {\scriptsize $\theta$};  \node[ver] () at (24.6, 20.6) {\scriptsize $\theta$};  
 \node[ver] () at (30.9, 19.8) {\scriptsize $\theta$};

\end{tikzpicture} 

\caption{In Case 3, when the vertices are on a hypercycle (shown in bold), each edge together with its projection to the equidistant geodesic $\ell$ defines a quadrilateral.} \label{fig:hypercycle}
\end{figure}

%\smallskip

\noindent {\sf Case 3.} $Q$ is inscribed in a hypercycle. Recall from Subsection 2.1 that there is a geodesic line that is equidistant from the hypercycle, we denote it by $\ell$. This time, for each of the three sides $v_{i-1}v_i, v_iv_{i+1}$ and $v_{i+1}v_{i+2}$, consider the hyperbolic quadrilateral formed by the side,  its nearest-point projection to $\ell$, and the two distance-minimizing arcs from the endpoints of the side to $\ell$ (see Fig. \ref{fig:hypercycle}).  If the vertices on the hypercycle are $v$ and $v^\prime$, we denote the resulting quadrilateral by $vv^\prime w^\prime w$ where $w$ and $w^\prime$ lie on $\ell$ and are the nearest-point projections of $v$ and $v^\prime$ respectively; note that the angle $\angle vww^\prime = \angle v^\prime w^\prime w = \pi/2$ since the arcs $vw$ and $v^\prime w^\prime$ minimize distances to $\ell$. For such a quadrilateral, there is a reflection isometry that simultaneously interchanges the pairs of vertices $v,v^\prime$ and $w, w^\prime$; this implies that $\angle wvv^\prime = \angle w^\prime v^\prime v$. Moreover, since there is a  hyperbolic isometry that acts by translation along $\ell$ and the equidistant horocycle,  any two such quadrilaterals determined by pairs of vertices on the hypercycle are congruent if the lengths of those sides are the same.  In particular, if $w_{i-1}, w_i, w_{i+1}$ and $w_{i+2}$ are the nearest-point projections of $v_{i-1}, v_i, v_{i+1}, v_{i+2}$ respectively, then the angles $\angle w_{i-1} v_{i-1} v_i = \angle w_i v_i v_{i-1} = \angle w_{i} v_{i} v_{i+1} = \angle w_{i+1} v_{i+1} v_{i}   = \angle w_{i+1} v_{i+1} v_{i+2} = \angle w_{i+2} v_{i+2} v_{i+1} = \theta$ which implies the angles $B=2 \theta = C$.

\medskip 

\noindent This proves Claim \ref{cl:h3}.

\medskip

By Claim \ref{cl:h3}, the interior angles of $X$ at the vertices $v_{i}$ and $v_{i+1}$ are same. Since $i$ is arbitrary, it follows that  the interior angles of $X$ at the vertices $v_{i}$ and $v_{i+1}$ are same for all $i = 1,  \dots, n$. (Here $v_{n+1}=v_1$.) This implies that $X$ is equiangular. Therefore, by Claim \ref{cl:h1}, $X$ is regular. 

\bigskip

 \noindent {\bf Spherical case:} 
Let $Y$ be a spherical $n$-gon in an open hemisphere $U$ of perimeter $L$ and of largest area among all spherical $n$-gons in $U$ of perimeters $L$. 

By similar arguments as in the proof of Claim \ref{cl:h1} (using Lhuilier's formula in place of Heron's formula and considering Fig. \ref{fig:triangle_s} in place of Fig. \ref{fig:triangle_h}), it follows that  $Y$ is equilateral.

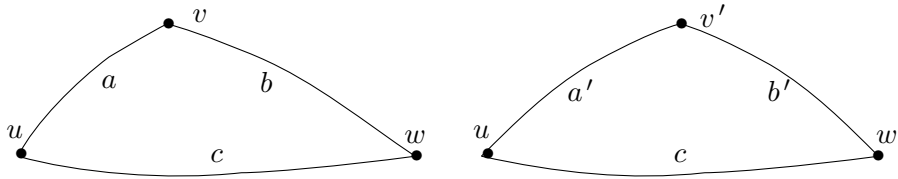
\begin{figure}[ht!]

\tikzstyle{ver}=[]
\tikzstyle{vert}=[circle, draw, fill=black!100, inner sep=0pt, minimum width=4pt]
\tikzstyle{vertex}=[circle, draw, fill=black!00, inner sep=0pt, minimum width=4pt]
\tikzstyle{edge} = [draw,thick,-]
\centering

\begin{tikzpicture}[scale=0.22]

\draw (8, 12) .. controls (9, 14) and (11.5, 16.5) .. (13.5, 18) .. controls (13.5, 18) and (16, 19.5) .. (17,20); 

\draw (32, 12) .. controls (29, 14) and (26, 16.5) .. (22.5, 18) .. controls (22.5, 18) and (19, 19.5) .. (17,20); 

\draw (8, 12) .. controls (11.5, 11) and (17, 10.5) .. (21.5, 11) .. controls (23, 11) and (27.5, 11.4) .. (32,12); 

\node[ver] () at (8.2, 12.1) {\small $\bullet$}; 
\node[ver] () at (17.1, 20) {\small $\bullet$};  
\node[ver] () at (32.1, 12) {\small $\bullet$}; 

 \node[ver] () at (7.8, 13.5) {$u$};  \node[ver] () at (19, 20.5) { $v$};  \node[ver] () at (32, 13.2) { $w$}; 
 
 \node[ver] () at (13.5, 16.5) {$a$};  \node[ver] () at (23, 16.5) {$b$};  
 \node[ver] () at (20, 12) {$c$}; 

%%%%------------------------

\draw (36, 12) .. controls (38, 14) and (40, 16) .. (42.5, 17.5) .. controls (42.5, 17.5) and (46, 19.5) .. (48,20); 

\draw (60, 12) .. controls (58, 14) and (56, 16) .. (53.5, 17.5) .. controls (53.5, 17.5) and (50, 19.5) .. (48,20); 

\draw (36, 12) .. controls (40.5, 11) and (45, 10.5) .. (49.5, 11) .. controls (51, 11) and (55.5, 11.4) .. (60,12); 

\node[ver] () at (36.4, 12.1) {\small $\bullet$}; 
\node[ver] () at (48.1, 20) {\small $\bullet$};  
\node[ver] () at (60, 12) {\small $\bullet$}; 

 \node[ver] () at (36, 13.5) {$u$};  \node[ver] () at (50, 20.5) { $v^{\,\prime}$};  \node[ver] () at (60.5, 13.2) { $w$}; 
 
\node[ver] () at (42, 16) {$a^{\,\prime}$};  \node[ver] () at (54, 16) {$b^{\,\prime}$};  \node[ver] () at (48, 12) {$c$}; 

\end{tikzpicture} 

\caption{Spherical triangles} \label{fig:triangle_s}
\end{figure}

%%%

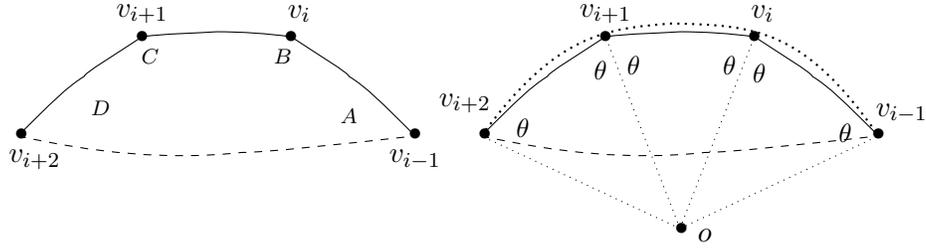
\begin{figure}[ht!]

\tikzstyle{ver}=[]
\tikzstyle{vert}=[circle, draw, fill=black!100, inner sep=0pt, minimum width=4pt]
\tikzstyle{vertex}=[circle, draw, fill=black!00, inner sep=0pt, minimum width=4pt]
\tikzstyle{edge} = [draw,thick,-]
\centering

\begin{tikzpicture}[scale=0.22]

\draw (8, 12) .. controls (10, 14) and (11, 15) .. (12, 15.6) .. controls (12, 15.8) and (14, 17) .. (15.5,18); 

\draw (32, 12) .. controls (30, 14) and (29, 15) .. (28, 15.6) .. controls (28, 15.8) and (26, 17) .. (24.5,18); 

\draw (15.5, 18) .. controls (17, 18.2) and (18.5, 18.2) .. (20, 18.3) .. controls (21.5, 18.3) and (23, 18.2) .. (24.5,18); 

\draw[dashed] (8, 12) .. controls (12.5, 11) and (17, 10.5) .. (21.5, 11) .. controls (23, 11) and (27.5, 11.5) .. (32,12); 

\node[ver] () at (8.2, 12.1) {\small $\bullet$}; 
\node[ver] () at (32, 12.1) {\small $\bullet$}; 
\node[ver] () at (15.5, 18) {\small $\bullet$}; 
\node[ver] () at (24.5, 18) {\small $\bullet$}; 

 \node[ver] () at (9, 10.6) {$v_{i+2}$};  \node[ver] () at (15.5, 19.4) { $v_{i+1}$}; 
 \node[ver] () at (25, 19.4) {$v_{i}$};  \node[ver] () at (32, 10.6) { $v_{i-1}$}; 
 
\node[ver] () at (13, 13.7) {\scriptsize $D$};  \node[ver] () at (16, 16.8) {\scriptsize  $C$};  \node[ver] () at (24, 16.8) {\scriptsize $B$};  \node[ver] () at (28, 13.2) {\scriptsize $A$}; 

%%%------------------------

\draw (36, 12) .. controls (38, 14) and (39, 15) .. (40, 15.6) .. controls (40, 15.8) and (42, 17) .. (43.5,18); 

\draw (60, 12) .. controls (58, 14) and (57, 15) .. (56, 15.6) .. controls (56, 15.8) and (54, 17) .. (52.5,18); 

\draw (43.5, 18) .. controls (45, 18.2) and (46.5, 18.2) .. (48, 18.3) .. controls (49.5, 18.3) and (51, 18.2) .. (52.5,18);

\draw[dashed] (36, 12) .. controls (40.5, 11) and (45, 10.5) .. (49.5, 11) .. controls (51, 11) and (55.5, 11.5) .. (60,12); 

\draw[dotted, thick] (36, 12) .. controls (40.5, 18.2) and (43.55, 18.5) .. (47.5, 18.8) .. controls (52.5, 18.7) and (55.5, 18.5) .. (60,12); 

\draw[dotted]  (48, 6.4) -- (36, 12); \draw[dotted]  (48, 6.4) -- (43.5, 18); 
 \draw[dotted]  (48, 6.4) -- (52.5, 18); \draw[dotted]  (48, 6.4) -- (60, 12);

\node[ver] () at (36.2, 12.1) {\small $\bullet$}; 
\node[ver] () at (60, 12.1) {\small $\bullet$}; 
\node[ver] () at (43.5, 18) {\small $\bullet$}; 
\node[ver] () at (52.5, 18) {\small $\bullet$}; 
\node[ver] () at (48.1, 6.4) {\small $\bullet$}; 

 \node[ver] () at (35, 14) {$v_{i+2}$};  \node[ver] () at (43.5, 19.4) { $v_{i+1}$}; 
 \node[ver] () at (53, 19.4) {$v_{i}$};  \node[ver] () at (61.4, 13.4) { $v_{i-1}$}; 
 \node[ver] () at (49.5, 6) {$o$}; 
 
 \node[ver] () at (38.5, 12.5) {\small $\theta$};  \node[ver] () at (43.2, 16) {\small  $\theta$};  \node[ver] () at (45.2, 16.3) {\small  $\theta$};  
 
 \node[ver] () at (51, 16.3) {\small $\theta$};  \node[ver] () at (52.8, 15.8) {\small $\theta$};  \node[ver] () at (58, 12.2) {\small$\theta$}; 
 
\end{tikzpicture} 
\vspace{-4mm}
\caption{Spherical quadrilaterals} \label{fig:quardangle_s}
\end{figure} 

%%%%%%%%%%%%%%%%% ?????????????/

Again, by similar arguments as in the hyperbolic case (using Propositions \ref{W1} \& \ref{W2})
in place of Propositions \ref{W1} \& \ref{W3}), $Y$ is equiangular.  (In this case, we have to consider spherical quadrilateral $Q$ in place of a hyperbolic quadrilateral and we have to consider only Case 1, namely, $Q$ is inscribed in a circle.  See Fig. \ref{fig:quardangle_s}.) 
Since $Y$ is equilateral, this implies that  $Y$ is regular. This completes the proof. 
\end{proof}

%\noindent {\bf Acknowledgements:}
%This work was done  when the first author was a faculty in the Department of Mathematics, Indian Institute of Science, %Bangalore.

\end{document}